\documentclass[a4paper,twoside,11pt]{article}
\usepackage[english]{babel}

\usepackage{amsmath,amsfonts,amssymb,amsthm}
\newtheorem{teo}{Theorem}
\newtheorem{lem}{Lemma}

 \title{{\bf   On Smoothness of    the Abel Equation Solution in Terms of  the Jacoby Series Coefficients  }}

\author{Maksim \,V.~Kukushkin   \\ \\
 \small  \textit{Moscow State University of Civil Engineering}\\\textit{\small\textit{Russia, Moscow, 129337,}}\\
 \small  \textit{Kabardino-Balkarian Scientific Center, RAS, }\\\textit{\small\textit{Russia, Nalchik, 360051, kukushkinmv@rambler.ru}} }
\date{}
\begin{document}

\maketitle

\begin{abstract}
In this paper we continue  the  investigation of  the  Abel equation with the right part belonging to a  Lebesgue weighted space. We have improved the previously known result - the uniqueness and existence theorem  formulated in terms of the Jacoby series coefficients that gives us an opportunity to find and classify a solution due  to an asymptotic of some  relation containing the Jacoby coefficients of the  right part. The new  main results are in the following: The conditions imposed on the parameters, under which the Abel equation has a unique solution represented by the series, are formulated; The relationship between the values of the parameters  and  the  solution smoothness is established. The independence between one of the parameters and the smoothness of the solution is proved. 
\end{abstract}

\section{Introduction}
In the beginning, let us remind that the so called mapping theorems for the Riemann-Liouville operator were firstly studied   by H. Hardy and Littlewoode, it was proved that
 $I^{\alpha}_{a+}:L_{p}\rightarrow L_{q},\, 1<p<1/\alpha,\,q<p/(1-\alpha p),\,\alpha\in(0,1).$ This proposition was afterwards clarified \cite{firstab_lit H-L1}  and nowadays is known as the Hardy-Littlewood theorem with limit index $I^{\alpha}_{a+}:L_{p}\rightarrow L_{q},\,q=p/(1-\alpha p).$ However there was an attempt to extend this theorem on some class of  weighted Lebesgue spaces defined as functional spaces endowed with the following norm
  $$
  \|f\|_{L_{p}(I\!,\,\omega)}:=\left\{\int\limits_{a}^{b}\left|f(x)\right|^{p}\omega(x)dx\right\}^{1/p},\,\omega(x):=(x-a)^{\beta}(b-x)^{\gamma},\,\beta,\gamma\in \mathbb{R},\,I:=(a,b).
  $$ In this dirrection the mathematicians such as  Rubin B.S. \cite{firstab_lit: Rubin},\cite{firstab_lit: Rubin 1},\cite{firstab_lit: Rubin 2}, Karapetyants N.K.
\cite{firstab_lit: Karapetyants N. K. Rubin B. S. 1},\cite{firstab_lit: Karapetyants N. K. Rubin B. S. 2}, Vakulov B.G. \cite{firstab_lit: Vaculov},   Samko S.G.  \cite{firstab_lit: Samko M. Murdaev},\cite{firstab_lit: Samko Vakulov B. G.}
  (the results of \cite{firstab_lit: Karapetyants N. K. Rubin B. S. 1},\cite{firstab_lit: Rubin},\cite{firstab_lit: Rubin 1} are also  presented in \cite{firstab_lit:samko1987}) had a great success. There were formulated the analogs of  Hardy-Littlewood theorem for a class of weighted Lebesgue spaces. The main disadvantage of the results presented in \cite{firstab_lit:samko1987} is gaps of the parameters values in the conditions, moreover the notorious problem related to $p=1/\alpha$ was remained completely unsolvable. All these create the prerequisite to invent another approach for studying the Riemann-Liouville operator action that was successfully investigated in the paper \cite{firstab_lit: Kuk ax.} and we write out bellow some of its highlights. In spite of the fact that  the idea of using the Jacoby polynomials is not novel and many papers were devoted to this topic
  \cite{firstab_lit Saad},\cite{firstab_lit:  Indian},\cite{firstab_lit:  Iranian society},\cite{firstab_lit:E.H. Doha},
  \cite{firstab_lit: Fr. oder Leg. functions},\cite{firstab_lit: SHENG CHEN},\cite{firstab_lit:  Abstract and Applied Analysis  ekim B} we confirm  the main advantage of the method, used in  the paper \cite{firstab_lit: Kuk ax.} and based on the results
  \cite{firstab_lit:H. Pollard 1},\cite{firstab_lit:H. Pollard 2},\cite{firstab_lit:H. Pollard 3},\cite{firstab_lit:Muckenhoupt},\cite{firstab_lit:H. Newman},  is still relevant and allows us to obtain some interesting results.     The main challenge of this paper is to improve and clarify the results of the paper
  \cite{firstab_lit: Kuk ax.}. In particular we need to find a simple  condition, on the right part of the Abel equation, under which Theorem 2 \cite{firstab_lit: Kuk ax.} is applicable. For this purpose we have made  an attempt to investigate this problem having used    absolute convergence of a series. The main relevance of the improvement is based on the fact that the previously used methods were  determined by the relation between order $\alpha$ of the fractional  integral and  index $p$ of a Lebesgue space (for instance the case $p=1/\alpha$ is not considered in the monograph \cite{firstab_lit:samko1987}). All these create a strong motivation  of researching  in this direction, but the highlight is in the following: The relationship between the values of the parameters   and order $\alpha$, by virtue of which we can provide a  description of the solution smoothness,    is established; The conditions  providing existence and uniqueness of the solution,   formulated  in terms of Jacoby series coefficients,   were obtained. The principal result - the   independence between one of the parameters and  the solution smoothness   was proved.

\section{Preliminaries }
  The orthonormal system of the  Jacobi  polynomials is denoted by
$$
p_{n}^{\,(\beta,\gamma)}(x)= \delta_{n} (\beta,\gamma) \, y^{\,(\beta,\gamma)}_{n}(x),\,n\in \mathbb{N}_{0},
$$
where  the  normalized multiplier $\delta_{n} (\beta,\gamma)$ is  defined by the formula
$$
\delta_{n}  (\beta,\gamma) =(-1)^{n}\frac{\sqrt{\beta+\gamma+2n+1}}{(b-a)^{n+(\beta+\gamma+1)/2}}\cdot \sqrt{\frac{ \Gamma(\beta+\gamma+n+1)}{n!\Gamma(\beta +n+1)\Gamma( \gamma+n+1)}} \;,\,
$$
$$
\delta_{0} (\beta,\gamma) =  \frac{1}{\sqrt{\Gamma(\beta  +1)\Gamma( \gamma +1)}}\,,\;\beta+\gamma+1=0,
$$
the orthogonal polynomials $y^{( \beta,\gamma)}_{n}$ are defined by the formula
$$
y^{( \beta,\gamma)}_{n}(x) =   (x-a)^{-\beta}(b-x)^{-\gamma}\frac{d^{n}}{dx^{n}}\left[(x-a)^{\beta+n}(b-x)^{\gamma+n}\right],\, \beta,\gamma>-1.
$$
Consider the orthonormal    Jacobi  polynomials
$$
p^{(\beta,\gamma)}_{n} (x)= \delta_{n }  y_{n} (x),\,\beta,\gamma>-1/2,\,n\in \mathbb{N}_{0}.
$$
It is clear that
$$
\!\!p_{n}^{(k)}(a) \! =\!\!\frac{(-1)^{n+k} }{(b-a)^{k+(\beta+\gamma+1)/2}} \cdot \sqrt{\frac{(\beta+\gamma+2n+1) \Gamma(\beta+\gamma+n+1)}{n!\Gamma(\beta +n+1)\Gamma( \gamma+n+1)}} \sum\limits_{i=0}^{k}C^{i}_{n}  \tbinom{n+\beta} {n-i}   \tbinom{n+\gamma} {i}  C^{i}_{k}    \tbinom{n-i}{k-i}  i!=
$$
$$
= (-1)^{n+k}(b-a)^{-k-(\beta+\gamma+1)/2} \delta'_{n}  \tilde{C}_{n}^{k}(\beta,\gamma) ,\,k\leq n,
$$
where
$$
\delta'_{n}:=\sqrt{\frac{(\beta+\gamma+2n+1) \Gamma(\beta+\gamma+n+1)}{n!\Gamma(\beta +n+1)\Gamma( \gamma+n+1)}},\,\tilde{C}_{n}^{k}(\beta,\gamma):= \sum\limits_{i=0}^{k}C^{i}_{n}  \tbinom{n+\beta} {n-i}   \tbinom{n+\gamma} {i}  C^{i}_{k}    \tbinom{n-i}{k-i}  i!.
$$
 In the same way, we get
$$
p_{n}^{(k)}(b)  =  (b-a)^{k+(\beta+\gamma+1)/2}\delta'_{n}\tilde{C}_{n}^{k}( \gamma,\beta)  ,\,k\leq n.
$$
  Using    the Taylor  series expansion  for the  Jacobi   polynomials, we get
$$
p^{(\beta,\gamma)}_{n}(x)= \sum\limits_{k=0}^{n}(-1)^{n+k}\frac{ \delta'_{n}\tilde{C}_{n}^{k}(\beta,\gamma)}{k!(b-a)^{k+(\beta+\gamma+1)/2}}  (x-a)^{k}=    \sum\limits_{k=0}^{n} (-1)^{ k}   \frac{ \delta'_{n}\tilde{C}_{n}^{k}( \gamma,\beta)}{k! (b-a)^{k+(\beta+\gamma+1)/2}}   (b-x)^{k} .
$$
 Applying the formulas (2.44),(2.45) of  the   fractional integral and   derivative  of a power  function  \cite[p.40]{firstab_lit:samko1987},  we obtain
$$
(I_{a+}^{\alpha}p_{n})(x)= \sum\limits_{k=0}^{n} (-1)^{n+k} \frac{ \delta'_{n}\tilde{C}_{n}^{k}(\beta,\gamma)}{ (b-a)^{k+(\beta+\gamma+1)/2}\Gamma(k+1+\alpha)}     (x-a)^{k+\alpha},
$$
$$
\;(I_{b-}^{\alpha}p_{n})(x)=
  \sum\limits_{k=0}^{n} (-1)^{ k} \frac{ \delta'_{n}\tilde{C}_{n}^{k}(\gamma,\beta)}{ (b-a)^{k+(\beta+\gamma+1)/2}\Gamma(k+1+\alpha)}  (b-x)^{k+\alpha},\,\alpha\in(-1,1),
$$
here we used the formal denotation $I^{-\alpha}_{a+}:=D_{a+}^{\alpha}.$
 Thus, using integration by parts, we get
$$
\int\limits_{a}^{b}p_{m}(x)(I^{\alpha}_{a+}p_{n})(x)\omega(x)dx=
$$
$$
      =(-1)^{n } \delta'_{m}\delta'_{n}\sum\limits_{k=0}^{n} (-1)^{ k}   \frac{  \tilde{C}_{n}^{k}(\beta,\gamma)B(\alpha+\beta+k+1,\gamma+m+1) }{\Gamma(k+\alpha-m+1)}.
$$
In the same way, we get
$$
(p_{m},I_{b-}^{\alpha}p_{n})_{L_{2}(I\!,\,\omega)}=(-1)^{m } \delta'_{m}\delta'_{n}\sum\limits_{k=0}^{n} (-1)^{ k}   \frac{  \tilde{C}_{n}^{k}(\gamma,\beta)B(\alpha+\gamma+k+1,\beta+m+1) }{\Gamma(k+\alpha-m+1)}.
$$
Using  the denotation
$$
A^{\alpha,\beta,\gamma}_{mn}:= \delta'_{m}\delta'_{n}\sum\limits_{k=0}^{n} (-1)^{ k}   \frac{  \tilde{C}_{n}^{k}(\beta,\gamma)B(\alpha+\beta+k+1,\gamma+m+1) }{\Gamma(k+\alpha-m+1)}\,,
$$
 we have
\begin{equation}\label{4}
 (p_{m},I_{a+}^{\alpha}p_{n})_{L_{2}(I\!,\,\omega)}=(-1)^{n}A^{\alpha,\beta,\gamma}_{mn},\;\;
 (p_{m},I_{b-}^{\alpha}p_{n})_{L_{2}(I\!,\,\omega)}=(-1)^{m}A^{\alpha,\gamma,\beta}_{mn}.
 \end{equation}
The following theorem is the very mapping theorem (see \cite{firstab_lit: Kuk ax.}) formulated in terms of the Jacoby series coefficients. Here we give  the modified form  corresponding to the right-side case.

\begin{teo}\label{T1}(Theorem 2 in \cite{firstab_lit: Kuk ax.})
Suppose  $ \omega(x)=(x-a)^{\beta}(b-x)^{\gamma},\, \beta,\gamma \in[-1/2,1/2],\,\alpha\in (-1,0] ,$ the Pollard condition holds
$$
 4\max\left\{\frac{\beta+1}{2\beta+3},\frac{\gamma+1}{2\gamma+3}\right\}<p< 4\min\left\{\frac{\beta+1}{2\beta+1},\frac{\gamma+1}{2\gamma+1}\right\},
$$
 the right part of    the Abel equation $I_{b-}^{-\alpha}\varphi=f$ such that
\begin{equation}\label{1}
 \left\|D^{-\alpha}_{b-} S_{k} f \right\|_{L_{p}(I\!,\,\omega)}\! \leq C,\;k\in \mathbb{N}_{0},\;\;\left|\sum\limits_{n=0}^{\infty}  f_{n}A^{\alpha,\gamma,\beta}_{mn}\right|\leq C m^{-\lambda},\; m\in \mathbb{N} ,\;\lambda\in[0,\infty);
\end{equation}
then  there exists  a   unique solution of  the Abel  equation  in     $L_{p}(I,\omega),$  the   solution belongs to   $L_{q}(I,\omega),$  where: $q=p,$  if $0\leq\lambda\leq 1/2 \,;\;q=\max\{p,t\},\,  t<(2s-1)/(s-\lambda),$ if $\,1/2<\lambda<s\;(s=3/2+\max\{\beta,\gamma\});$  $q$ is arbitrary large, if $\lambda\geq\,s.$
 Moreover if $\lambda>1/2,$ then the solution   is represented by the convergent in $L_{q}(I,\omega)$   series
\begin{equation}\label{2}
\psi(x)=\sum\limits_{m=0}^{\infty}p_{m}(x) (-1)^{m}\sum\limits_{n=0}^{\infty} f_{n}A^{\alpha,\gamma,\beta}_{mn}.
\end{equation}
 \end{teo}

   We also  need an adopted version (see \cite{firstab_lit: Kuk ax.}) of the    Zigmund-Marczinkevich theorem  (see \cite{firstab_lit: Marz}), which  establishes the following.
\begin{teo}\label{T2}
  If $q\geq2$ and we have
\begin{equation}\label{3}
\Omega_{q}(c)=\left(\sum\limits_{n=1}^{\infty}|c_{n}|^{q} n^{(\max\{\beta,\gamma\}+3/2)(q-2)}\right)^{1/q}<\infty,\,\max\{\beta,\gamma\}\geq-1/2,
\end{equation}
then the series
$$
\sum_{n=1}^{\infty}c_{n}p^{(\beta,\gamma)}_{n}(x)
$$
converges in  $ L_{q}(I,\omega),\,\omega(x)=(x-a)^{\beta}(b-x)^{\gamma}$ to some function  $f\in L_{q}(I,\omega)$  and $\|f\|_{L_{q}(I,\omega)}\leq C\Omega_{q}(c).$
\end{teo}

\section{Main results}
Everywhere further, in the contrary to the paper \cite{firstab_lit: Kuk ax.},  we consider the right-side case, assuming that  $\alpha\in(-1,0],$ but the  reasonings corresponding to the left-side case are   absolutely analogous.
\begin{lem}\label{L1} Suppose $k<m,$  
$$
I_{mk}:= \delta'_{m}\frac{ \Gamma(\beta+m+1)  \prod\limits_{i=1}^{m-k}(m-k-\alpha-i) }{ \Gamma(\alpha+\beta+k+\gamma+m+2)  };
$$
then the following estimates hold
$$
 I_{mk}\leq C  m^{  -2\alpha -\gamma-3/2  },\,I_{mk}\leq Ce^{2k} m^{ 2\xi-2\alpha -\gamma-5/2-2k },\,k=0,1,...,m-1,
$$
where $ \xi=0.577215...$ is the Mascheroni constant.
\end{lem}
\begin{proof}
 Consider
$$
I_{mk}= \sqrt{\frac{(\beta+\gamma+2m+1) \Gamma(\beta+\gamma+m+1)}{m!\Gamma(\beta +m+1)\Gamma( \gamma+m+1)}}\cdot\frac{ \Gamma(\beta+m+1)  \prod\limits_{i=1}^{m-k}(m-k-\alpha-i) }{ \Gamma(\alpha+\beta+k+\gamma+m+2)  }=
$$
$$
 = \sqrt{\frac{(\beta+\gamma+2m+1) \Gamma(\beta+\gamma+m+1)\Gamma(\beta+m+1)}{m! \Gamma( \gamma+m+1)}}\cdot\frac{    \Gamma(m-k-\alpha ) }{ \Gamma(-\alpha) \Gamma(\alpha+\beta+k+\gamma+m+2)  }<
$$
$$
< \sqrt{\frac{(\beta+\gamma+2m+1) \Gamma(\beta+\gamma+m+1)\Gamma(\beta+m+1)}{m! \Gamma( \gamma+m+1)}}\cdot\frac{    \Gamma(m -\alpha ) }{ \Gamma(-\alpha) \Gamma(\alpha+\beta +\gamma+m+2)  }.
$$
Now we should take into account the following relations (see \cite[p.xv]{firstab_lit:Hardy})
\begin{equation}\label{5}
\frac{\Gamma(m+\beta+1)}{m!}=m^{\beta}+\sum\limits_{s=1}^{p}c_{s}m^{\beta-p}+O(m^{\beta-p-1}),\, \beta\neq-1,-2,...
\end{equation}
 Having applied   formula \eqref{5},    we obtain
$$
 \frac{\Gamma(\beta+\gamma+m+1)}{\Gamma( \beta+m+1)}=\frac{m^{\beta+\gamma}+\sum\limits_{s=1}^{p}c_{s1}m^{\beta+\gamma-s}+O(m^{\beta+\gamma-p-1})}{m^{\beta}+\sum\limits_{s=1}^{p}c_{s2}m^{\beta-s}+O(m^{\beta-p-1})}=
$$
$$
  =m^{\gamma}\frac{m^{\beta }+\sum\limits_{s=1}^{p}c_{s1}m^{\beta -s}+O(m^{\beta+\gamma-p-1})n^{-\gamma}}{m^{\beta}+\sum\limits_{s=1}^{p}c_{s2}m^{\beta-s}+O(m^{\beta-p-1})}=
 m^{\gamma}\frac{m^{\beta }+\sum\limits_{s=1}^{p}c_{s1}m^{\beta -s}+O(m^{\beta -p-1}) }{m^{\beta}+\sum\limits_{s=1}^{p}c_{s2}m^{\beta-s}+O(m^{\beta-p-1})}=
$$
$$
 = m^{\gamma} +m^{\gamma}\frac{\sum\limits_{s=1}^{p}(c_{s1}-c_{s2})m^{\beta-s} +O(m^{\beta-p-1})- O(m^{\beta-p-1})}{m^{\beta}+\sum\limits_{s=1}^{p}c_{s2}m^{\beta-s}+O(m^{\beta-p-1})}=
  $$
\begin{equation}\label{6}
  =m^{\gamma} +m^{\gamma}\frac{\sum\limits_{s=1}^{p}(c_{s1}-c_{s2})m^{\beta-s} +   O(m^{\beta-p-1})}{m^{\beta}+\sum\limits_{s=1}^{p}c_{s2}m^{\beta-s}+O(m^{\beta-p-1})}\sim m^{\gamma},
\end{equation}
since the notation $f(m):=O(m^{\beta-p-1})$ implies that there exists such a constant $C$ so that $|f(m)|\leq C  m^{\beta-p-1}.$
Using  formula \eqref{6}, we get
$$
\frac{    \Gamma(m -\alpha ) }{   \Gamma(\alpha+\beta +\gamma+m+2)  }\sim m^{-2\alpha-\beta -\gamma-2}.
$$
In an analogous way, we have
$$
\sqrt{\frac{(\beta+\gamma+2m+1) \Gamma(\beta+\gamma+m+1)\Gamma(\beta+m+1)}{m! \Gamma( \gamma+m+1)}}\sim m^{\beta-1/2}.
$$
Combining these two relations, we obtain the first estimate this theorem has claimed.
However this estimate can be improved for sufficiently large values $m$ and $k.$   To manage such a result
 we should take into account the following relation (see\cite{firstab_lit: X.Li ineq})
$$
\frac{x^{x-\xi}}{e^{x-1}}<\Gamma(x)<\frac{x^{x-1/2}}{e^{x-1}},\,x>1,\,  \xi=0.577215...\, .
$$
Having taken into account this formula, we can estimate
$$
\frac{\Gamma(x+\delta )}{\Gamma(x )}<e^{-\delta}\frac{(x+\delta)^{x+\delta-1/2}}{x^{x-\xi}}=e^{-\delta}\frac{(x+\delta)^{x+\delta-1/2}}{x^{x+\delta-1/2}}x^{\xi+\delta-1/2}=
e^{-\delta} \left(1+ \frac{\delta}{x}\right)^{x+\delta-1/2}  x^{\xi+\delta-1/2}=
$$
$$
=e^{-\delta} \left(1+ \frac{\delta}{x}\right)^{\frac{x}{\delta}\cdot \frac{x+\delta-1/2}{x}\cdot\delta}  x^{\xi+\delta-1/2}\sim x^{\xi+\delta-1/2},\,x\rightarrow\infty.
$$
In an analogous way, it is not hard to prove the following estimate
$$
\frac{\Gamma(x+\delta )}{\Gamma(x )}>e^{-\delta}\frac{(x+\delta)^{x+\delta-\xi}}{x^{x-1/2}}=e^{-\delta}\frac{(x+\delta)^{x+\delta-\xi}}{x^{x+\delta-\xi}}x^{1/2+\delta-\xi}=
e^{-\delta} \left(1+ \frac{\delta}{x}\right)^{x+\delta-\xi}  x^{1/2+\delta-\xi}=
$$
$$
=e^{-\delta} \left(1+ \frac{\delta}{x}\right)^{\frac{x}{\delta}\cdot \frac{x+\delta-\xi}{x}\cdot\delta}  x^{1/2+\delta-\xi} \sim x^{1/2+\delta-\xi},\,x\rightarrow\infty.
$$
Using these formulas we have
$$
J:= \frac{  \Gamma(-\alpha )  \prod\limits_{i=1}^{m-k}(m-k-\alpha-i) }{ \Gamma(\alpha+\beta+k+\gamma+m+2)  }=\frac{    \Gamma(m-k-\alpha ) }{   \Gamma(\alpha+\beta+k+\gamma+m+2)  }<
$$
$$
<\frac{e^{-\delta_{1}} \left(1+ \frac{\delta_{1}}{m}\right)^{m+\delta_{1}-1/2}  m^{\xi+\delta_{1}-1/2}}{e^{-\delta_{2}} \left(1+ \frac{\delta_{2}}{m}\right)^{m+\delta_{2}-\xi}  m^{1/2+\delta_{2}-\xi}}=
 e^{-\delta_{1}+\delta_{2}}  \left(1+ \frac{\delta_{1}}{m}\right)^{m+\delta_{1}-1/2}    \left(1+ \frac{\delta_{2}}{m}\right)^{-m-\delta_{2}+\xi}    m^{2\xi-1+\delta_{1}-\delta_{2} },
$$
where
$
\delta_{1}=-k-\alpha,\,\delta_{2}=k+\alpha+\beta+\gamma +2.
$
Note that for concrete $\delta_{1},\delta_{2}$ we have the following tending
$$
\left(1+ \frac{\delta_{1}}{m}\right)^{m+\delta_{1}-1/2}\rightarrow e^{\delta_{1}} ,\,\left(1+ \frac{\delta_{2}}{m}\right)^{-m-\delta_{2}+\xi}\rightarrow e^{-\delta_{2}},\,m\rightarrow\infty.
$$
Hence
$$
J< Ce^{-\delta_{1}+\delta_{2}}  m^{2\xi-1+\delta_{1}-\delta_{2} }=Ce^{2k+2\alpha+\beta+\gamma +2}  m^{2\xi-1-(2k+2\alpha+\beta+\gamma +2) }.
$$
 Using formula \eqref{6}, we have
\begin{equation*} 
 \delta'_{m}  \Gamma(\beta+m+1)=\sqrt{\frac{(\beta+\gamma+2m+1) \Gamma(\beta+\gamma+m+1)\Gamma(\beta+m+1)}{m! \Gamma( \gamma+m+1)}}\leq Cm^{\beta+1/2}.
\end{equation*}
Combining these results, we obtain
$$
I_{mk}<Ce^{2k+2\alpha+\beta+\gamma +2}  m^{2\xi-1-(2k+2\alpha +\gamma +3/2) },\,k=0,1,...,m-1.
$$
Thus, this estimate has proved the claimed result.
\end{proof}

\begin{lem}\label{L2} Suppose
$$
 d_{k}(\eta):=\eta^{k}\sum\limits_{i=0}^{k}  \frac{ 2^{i}}{i! (k-i)! \Gamma(\gamma+i+1)},\,\eta\in \mathbb{N};
$$
then 
$$
\forall \eta  ,\,\exists N : d_{k+1}(\eta)<d_{k}(\eta),\,k>N .
$$

 \end{lem}
\begin{proof}
Assume that $k>\eta$ and consider the following relation
$$
\eta^{k+1}\sum\limits_{i=0}^{k+1}  \frac{ 2^{i}}{i! (k+1-i)! \Gamma(\gamma+i+1)}-\eta^{k}\sum\limits_{i=0}^{k}  \frac{ 2^{i}}{i! (k-i)! \Gamma(\gamma+i+1)}=
$$
$$
=\eta^{k}\sum\limits_{i=0}^{k}  \frac{ 2^{i}(\eta-1+i-k)}{i! (k+1-i)! \Gamma(\gamma+i+1)}+\frac{\eta^{k+1}2^{ k+1 }}{(k+1)!\Gamma(\gamma+k+2)}=
$$
$$
=\eta^{k} \left\{\sum\limits_{i=0}^{k-\eta }  \frac{ 2^{i}(\eta-1+i-k)}{i! (k+1-i)! \Gamma(\gamma+i+1)}+\sum\limits_{i=k-\eta+1}^{k}  \frac{ 2^{i}(\eta-1+i-k)}{i! (k+1-i)! \Gamma(\gamma+i+1)}\right\}+
 \eta^{k}\frac{\eta2^{  k+1 }}{(k+1)!\Gamma(\gamma+k+2)} =
$$
$$
=\eta^{k} \left\{\sum\limits_{i=\eta }^{k-\eta }  \frac{ 2^{i}(\eta-1+i-k)}{i! (k+1-i)! \Gamma(\gamma+i+1)}+\!\!\!\!\sum\limits_{i=k-\eta+1}^{k}  \frac{ 2^{i}(\eta-1+i-k)}{i! (k+1-i)! \Gamma(\gamma+i+1)}+\!\!\!\sum\limits_{i=0}^{\eta-1}  \frac{ 2^{i}(\eta-1+i-k)}{i! (k+1-i)! \Gamma(\gamma+i+1)}\right\}+
$$
$$
+\eta^{k}\frac{\eta2^{  k+1 }}{(k+1)!\Gamma(\gamma+k+2)} =
 \eta^{k}  \sum\limits_{i=\eta }^{k-\eta }  \frac{ 2^{i}(\eta-1+i-k)}{i! (k+1-i)! \Gamma(\gamma+i+1)}+
 $$
 $$
 + \eta^{k}\left\{\sum\limits_{i=k-\eta+2}^{k}  \frac{ 2^{i}(\eta-1+i-k)}{i! (k+1-i)! \Gamma(\gamma+i+1)}+\!\!\!\sum\limits_{i=1}^{\eta-1}  \frac{ 2^{i}(\eta-1+i-k)}{i! (k+1-i)! \Gamma(\gamma+i+1)} \right\}+
$$
$$
+\eta^{k}\left\{\frac{\eta2^{  k+1 }}{(k+1)!\Gamma(\gamma+k+2)}+   \frac{  (\eta-1 -k)}{  (k+1 )! \Gamma(\gamma +1)} \right\}=S_{1}+S_{2}+S_{3};
$$
Note that $S_{1},S_{3}$   are negative for a sufficiently large value $k.$
Consider separately the expression
$$
 S_{2}=\sum\limits_{i=k-\eta+2}^{k}  \frac{ 2^{i}(\eta-1+i-k)}{i! (k+1-i)! \Gamma(\gamma+i+1)}+\!\!\!\sum\limits_{i=1}^{\eta-1}  \frac{ 2^{i}(\eta-1+i-k)}{i! (k+1-i)! \Gamma(\gamma+i+1)}=
$$
$$
=\sum\limits_{i=k-\eta+2}^{k-1}  \frac{ 2^{i}(\eta-1+i-k)}{i! (k+1-i)! \Gamma(\gamma+i+1)}+\!\!\!\sum\limits_{i=2}^{\eta-1}  \frac{ 2^{i}(\eta-1+i-k)}{i! (k+1-i)! \Gamma(\gamma+i+1)}+
$$
$$
+\frac{ 2^{k}(\eta-1)}{k!   \Gamma(\gamma+k+1)}+\frac{ 2 (\eta -k)}{   k ! \Gamma(\gamma+2)}=
 \sum\limits_{i=k-\eta+2}^{k-2}  \frac{ 2^{i}(\eta-1+i-k)}{i! (k+1-i)! \Gamma(\gamma+i+1)}+\!\!\!\sum\limits_{i=3}^{\eta-1}  \frac{ 2^{i}(\eta-1+i-k)}{i! (k+1-i)! \Gamma(\gamma+i+1)}+
$$
$$
+\frac{ 2^{k}(\eta-1)}{k!   \Gamma(\gamma+k+1)}+\frac{ 2 (\eta -k)}{   k ! \Gamma(\gamma+2)}+
\frac{ 2^{k-1}(\eta-2)}{(k-1)!  2! \Gamma(\gamma+k )}+\frac{ 2^{2} (\eta+1 -k)}{  2! (k-1) ! \Gamma(\gamma+3)}=
$$
$$
 =\sum\limits_{i=0}^{\eta-2} \left\{\frac{ 2^{k-i}(\eta-1-i)}{(k-i)!(i+1)!   \Gamma(\gamma+k+1-i)}+\frac{ 2^{i+1} (\eta -k+i)}{  (i+1)!(k-i)!  \Gamma(\gamma+2+i)}\right\}\leq
$$
$$
 \leq\sum\limits_{i=0}^{\eta-2} \left\{\frac{ 2^{k-i}(\eta-1-i)}{(k-i)!(i+1)!   \Gamma(\gamma+3)2^{k-2-i}}+\frac{ 2  (2\eta -k-2)}{  (i+1)!(k-i)!  \Gamma(\gamma+\eta)}\right\}<
$$
$$
 <\sum\limits_{i=0}^{\eta-2} \left\{\frac{  4(\eta-1 )}{(k-i)!(i+1)!   \Gamma(\gamma+3) }+\frac{ 2  (2\eta -k-2)}{  (i+1)!(k-i)!  \Gamma(\gamma+\eta)}\right\}.
$$
Now, it is clear that
$
 S_{2}<0
$
for a sufficiently large value $k.$ Combining this fact with the previously established fact regarding $S_{1},S_{3},$ we obtain the desired result. 
\end{proof}

\begin{lem}\label{L3} Assume that the following series is absolutely convergent
\begin{equation}\label{7}
 \sum\limits_{n=0}^{\infty} f_{n}  c_{n}   <\infty,
\end{equation}
where $f_{n} \in \mathbb{R},$
$$
 c_{n}=\frac{\delta'_{n}  n!\Gamma(n+\beta+1)\Gamma(n+\gamma+1)}{4^{ n }},
$$
then
$$
\sum\limits_{n=0}^{\infty}\left|f_{n}A_{mn}\right|\leq C  m^{-2\alpha -\gamma-3/2},\,m=1,2,...\,.
$$
\end{lem}
\begin{proof}
It is easy to see that
$$
A^{\alpha,\gamma,\beta}_{mn}:= \delta'_{m}\delta'_{n}\sum\limits_{k=0}^{n} (-1)^{ k}   \frac{  \tilde{C}_{n}^{k}( \gamma,\beta)B(\alpha+\gamma+k+1,\beta+m+1) }{\Gamma(k+\alpha-m+1)}=
$$
$$
= \delta'_{m}\delta'_{n}\sum\limits_{k=0}^{\min\{m-1,n\}} (-1)^{ k}   \frac{  \tilde{C}_{n}^{k}(\beta,\gamma)B(\alpha+\gamma+k+1,\beta+m+1) }{\Gamma(k+\alpha-m+1)}+
$$
$$
+\delta'_{m}\delta'_{n}\sum\limits_{k=m }^{n} (-1)^{ k}   \frac{  \tilde{C}_{n}^{k}(\beta,\gamma)B(\alpha+\gamma+k+1,\beta+m+1) }{\Gamma(k+\alpha-m+1)}=I_{1}+I_{2}.
$$
Consider $I_{1}.$
We need the following  formula (see (1.66) \cite{firstab_lit:samko1987})
$$
\Gamma(z)=\frac{\Gamma(z+n)}{z(z+1)...(z+n-1)},\,\mathrm{Re}z>-n.
$$
Denote $\alpha+1-[m-k]=:z,$ then for the case $m> k,$ we have $z>-[m-k].$ Hence
we obtain
$$
I_{1}= \delta'_{m}\delta'_{n}\sum\limits_{k=0}^{\min\{m-1,n\}} (-1)^{ k}   \frac{  \tilde{C}_{n}^{k}(\gamma,\beta)B(\alpha+\gamma+k+1,\beta+m+1) }{\Gamma(\alpha+1-[m-k])}=
$$
$$
 = (-1)^{m }\delta'_{m}\delta'_{n}\sum\limits_{k=0}^{\min\{m-1,n\}}     \frac{  \tilde{C}_{n}^{k}(\gamma,\beta)(m-k-\alpha-1)(m-k-\alpha-2)...(-\alpha) }{B^{-1}(\alpha+\gamma+k+1,\beta+m+1)\Gamma(1+\alpha) }=
$$
$$
 = (-1)^{m}\delta'_{m}\delta'_{n}\sum\limits_{k=0}^{\min\{m-1,n\}}     \frac{  \tilde{C}_{n}^{k}(\gamma,\beta)\prod\limits_{i=1}^{m-k}(m-k-\alpha-i) }{B^{-1}(\alpha+\gamma+k+1,\beta+m+1)\Gamma(1+\alpha) }=
$$
$$
 = (-1)^{m}\delta'_{m}\delta'_{n}\sum\limits_{k=0}^{\min\{m-1,n\}}  \tilde{C}_{n}^{k}(\gamma,\beta)   \frac{\Gamma(\alpha+\gamma+k+1)\Gamma(\beta+m+1)  \prod\limits_{i=1}^{m-k}(m-k-\alpha-i) }{ \Gamma(\alpha+\beta+k+\gamma+m+2)\Gamma(1+\alpha) }=
$$
$$
 = (-1)^{m} \delta'_{n}\sum\limits_{k=0}^{\min\{m-1,n\}}      \frac{\tilde{C}_{n}^{k}(\gamma,\beta)\Gamma(\alpha+\gamma+k+1)}{\Gamma(\alpha +1)}I_{mk}.
$$
Besides, having noticed that $\Gamma(\beta+k)\geq\Gamma(\beta+3)2^{k-3},\,k=3,4,...,$ we have the following reasonings
$$
\tilde{C}_{n}^{k}(\gamma,\beta) =\sum\limits_{i=0}^{k}C^{i}_{n}  \tbinom{n+\beta}{i} \tbinom{n+\gamma} {n-i}  C^{i}_{k} \tbinom{n-i}{k-i}   i!=
$$
$$
  =\Gamma(n+\beta+1)\Gamma(n+\gamma+1)\sum\limits_{i=0}^{k}  \frac{C^{i}_{n}C^{i}_{k}\Gamma(n-i+1)i!}{\Gamma(n+\beta-i+1)\Gamma(\gamma+i+1)\Gamma(n-k+1)}=
$$
$$
  =\frac{n!k!\Gamma(n+\beta+1)\Gamma(n+\gamma+1)}{\Gamma(n-k+1)}\sum\limits_{i=0}^{k}  \frac{ 1}{i! (k-i)!\Gamma(n+\beta-i+1)\Gamma(\gamma+i+1)}\leq
$$
$$
 \leq\frac{n!k!\Gamma(n+\beta+1)\Gamma(n+\gamma+1)}{\Gamma(\beta+3) 2^{n-k}}\sum\limits_{i=0}^{k}  \frac{ 2^{3}}{i! (k-i)!2^{n-i}\Gamma(\gamma+i+1)}=
$$
$$
 =\frac{ n!k!\Gamma(n+\beta+1)\Gamma(n+\gamma+1)}{\Gamma(\beta+3)2^{2n-3}}\sum\limits_{i=0}^{k}  \frac{ 2^{i+k}}{i! (k-i)! \Gamma(\gamma+i+1)}.
$$
Using Lemma \ref{L2}, we have  there exists such a constant $C>0$ so that
$$
\sum\limits_{i=0}^{k}  \frac{ 2^{i+k}}{i! (k-i)! \Gamma(\gamma+i+1)}\leq C e^{-2k} .
$$
Hence
$$
\tilde{C}_{n}^{k}(\gamma,\beta)\leq C\frac{ n!k!\Gamma(n+\beta+1)\Gamma(n+\gamma+1)}{\Gamma(\beta+3)4^{  n}e^{ 2k}}.
$$
Therefore, using Lemma \ref{L1}, we have
$$
|I_{1}|\leq C   \delta'_{n}\frac{ n!\Gamma(n+\beta+1)\Gamma(n+\gamma+1)}{ 4^{n }}\times
$$
$$
\times  \left\{ \Gamma(\alpha+\beta +1)      m^{ -2\alpha -\gamma -3/2  }+\sum\limits_{k=1}^{\min\{m-1,n\}}
\frac{k! \Gamma(\alpha+\beta+k+1)}{m^{2k-1}  }      m^{2\xi-2-2\alpha -\gamma -3/2  }\right\}.
$$
Consider
$$
\Xi_{m-1}:=\sum\limits_{k=1}^{ m-1 }
\frac{k! \Gamma(\alpha+\beta+k+1)}{m^{2k-1}  },
$$
we have
$$
\Xi_{m }-\Xi_{m-1}=\sum\limits_{k=1}^{ m-1 }
\frac{k! \Gamma(\alpha+\beta+k+1)[m^{2k-1}-(m+1)^{2k-1}]}{m^{2k-1}(m+1)^{2k-1}  }+\frac{m! \Gamma(\alpha+\beta+m+1)}{(m+1)^{2m-1}  }.
$$
Having applied the asymptotic Stirling formula (1.63) \cite[p.16]{firstab_lit:samko1987}, we obtain
$$
\frac{m! \Gamma(\alpha+\beta+m+1)}{(m+1)^{2m-1}  }\rightarrow 0,\,m\rightarrow\infty.
$$
Hence $\Xi_{m }-\Xi_{m-1}<0$ for sufficiently large $m.$ It implies that
$$
|I_{1}|\leq C   \delta'_{n}\frac{ n!\Gamma(n+\beta+1)\Gamma(n+\gamma+1)}{ 4^{n }}     m^{ -2\alpha -\gamma -3/2  } ;
$$
Consider $I_{2},$ we have the following reasonings
$$
|I_{2}|\leq \delta'_{m}\delta'_{n}\sum\limits_{k=m}^{n}    \frac{  \tilde{C}_{n}^{k}(\beta,\gamma)B(\alpha+\gamma+k+1,\beta+m+1) }{\Gamma(k+\alpha-m+1)}\leq
$$
$$
\leq\frac{ \delta'_{m}\delta'_{n} n! \Gamma(n+\beta+1)\Gamma(n+\gamma+1)}{\Gamma(\beta+3)2^{2n-3}}\sum\limits_{k=m}^{n}    \frac{ k!d_{k} B(\alpha+\gamma+k+1,\beta+m+1) }{2^{k}\Gamma(k+\alpha-m+1)}=
$$
$$
=\frac{ \delta'_{m}\delta'_{n} n! \Gamma(n+\beta+1)\Gamma(n+\gamma+1)}{\Gamma(\beta+3)2^{2n-3}}\sum\limits_{k=m}^{n}
  \frac{d_{k} k! \Gamma(\alpha+\gamma+k+1) \Gamma(m+\beta+1)}{2^{k}\Gamma(\alpha+\gamma+k +\beta+m+2)\Gamma(k+\alpha-m+1)}=
$$
where
$$
d_{k}:=4^{k}\sum\limits_{i=0}^{k}  \frac{ 2^{i}}{i! (k-i)! \Gamma(\gamma+i+1)},
$$
and we know, as it was proved, that $d_{k+1}<d_{k}$ for sufficiently large $k.$
Haven taken into account all these facts, using  the D'Alembert's principle,  it is easy to show that the following series is convergent and there exists such a constant $C>0$ so that
$$
\sum\limits_{k=m}^{\infty}    \frac{d_{k} k! \Gamma(\alpha+\gamma+k+1)  }{2^{k}\Gamma(\alpha+\gamma+k +\beta+m+2)\Gamma(k+\alpha-m+1)} <C,\, m=1,2,...\,.
$$
Moreover, since
$$
\frac{ \Gamma(m+\beta+2) }{ \Gamma(\alpha+\gamma+k +\beta+m+3) }=\frac{ \Gamma(m+\beta+1) (m+\beta+1) }{ \Gamma(\alpha+\gamma+k +\beta+m+2)(\alpha+\gamma+k +\beta+m+2) }<
$$
$$
<\frac{ \Gamma(m+\beta+1) }{ \Gamma(\alpha+\gamma+k +\beta+m+2) },\,k=m,m+1,...\,  ,
$$
then
$$
\sum\limits_{k=m}^{\infty}    \frac{d_{k} k! \Gamma(\alpha+\gamma+k+1)\Gamma(m+\beta+1) }{2^{k}\Gamma(\alpha+\gamma+k +\beta+m+2)\Gamma(k+\alpha-m+1)} <C,\, m=0,1,...\,.
$$
Hence
$$
|I_{2}|\leq C\frac{ \delta'_{m}\delta'_{n} n! \Gamma(n+\beta+1)\Gamma(n+\gamma+1)}{\Gamma(\beta+3)4^{n }}\leq C
\delta'_{n}\frac{  n! \Gamma(n+\beta+1)\Gamma(n+\gamma+1)}{ 4^{n }}\cdot\frac{m^{1/2+\beta}}{\Gamma(m+\beta+1)}\leq
$$
$$
\leq C
\delta'_{n}\frac{  n! \Gamma(n+\beta+1)\Gamma(n+\gamma+1)}{ 4^{n }} m^{-2\alpha -\gamma-3/2 },\,m=1,2,...\,.
$$
Due to the absolute convergence of the series  \eqref{7} we can extract a multiplier in each term of the series i.e.
$$
\sum\limits_{n=0}^{\infty}\left|f_{n}A_{mn}\right|\leq C\sum\limits_{n=0}^{\infty} |f_{n}|  c_{n}\leq C m^{-2\alpha -\gamma-3/2 },\,m=1,2,...\,.
$$
\end{proof}

\begin{teo} Assume that the Jacoby coefficients $f_{n}$ of the right part of the  Abel equation $I^{-\alpha}_{b-}\varphi=f$  such that corresponding series $\eqref{7}$ is absolutely convergent, the condition  $ 2\alpha +\gamma+1>0$ holds, then there exists a unique solution  of the Abel equation, the solution is represented by series \eqref{2} and has smoothness in accordance with  Theorem \ref{T2}, where $\lambda=2\alpha +\gamma+3/2.$
 \end{teo}

\begin{proof}
Due to Lemma \ref{L3}, we have
$$
\left|\sum\limits_{n=0}^{\infty} f_{n}A^{\alpha,\gamma,\beta}_{mn}\right|\leq C  m^{-2\alpha -\gamma-3/2},\,m=1,2,...\,.
$$
and it is clear that  $\lambda>1/2,$ since $ 2\alpha +\gamma+1>0.$
Thus to fulfill the conditions of Theorem \ref{T1}, we must show that
$$
\|D^{-\alpha}_{b-}S_{k}f\|_{L_{p}}\leq C,\, k=0,1,...\,,
$$
where $p$ is such an index value that the Polard condition holds. 
Consider
$$
\left(D^{-\alpha}_{b-}S_{k}f,p_{m}\right)=\sum\limits_{n=0}^{k}f_{n}\left(D^{-\alpha}_{b-}p_{n},p_{m}\right)= (-1)^{m}\sum\limits_{n=0}^{k}f_{n}A^{\alpha,\gamma,\beta}_{mn}=:c_{mk}\,.
$$
Let us impose  the conditions  on $\alpha, \gamma$  under which  the following estimate holds
\begin{equation}\label{8}
\|D^{-\alpha}_{b-}S_{k}f\|_{L_{q}}\leq C \left(\sum\limits_{m=1}^{\infty}|c_{mk}|^{q} m^{(\max\{\beta,\gamma\}+3/2)(q-2)}\right)^{1/q}<\infty ,\,q\geq2.
\end{equation}
Using Lemma \ref{1} and calculating  powers, we have the following sufficient conditions of  series \eqref{8} convergence    in terms of Theorem 1 : $q<(2s-1)/(s-\lambda).$ 
   Regarding   the fulfilment of the Polard conditions, we should notice that
   $$
    4\max\left\{\frac{\beta+1}{2\beta+3},\frac{\gamma+1}{2\gamma+3}\right\}=\frac{2s-1}{s} 
   $$
   (this is the consequence of the fact that the function $z(x):= (x+1)/ (2x+3)$ is an increasing function), hence   we can choose $p$ so that 
$$
4\max\left\{\frac{\beta+1}{2\beta+3},\frac{\gamma+1}{2\gamma+3}\right\}<p<q.
$$
Combining these facts and using the H\"{o}lder inequality,   we come to the conclusion 
\begin{equation*} 
\|D^{-\alpha}_{b-}S_{k}f\|_{L_{p}} \leq C \left(\sum\limits_{m=1}^{\infty}|c_{mk}|^{q} m^{(\max\{\beta,\gamma\}+3/2)(q-2)}\right)^{1/q}<C  ,\,k=0,1,...\,,
\end{equation*}
where index $p$  satisfies the Polard conditions. 
Hence the conditions of Theorem \ref{T1} is fulfilled. This gives us the desired result.
\end{proof}

\end{document}